\documentclass[12pt]{amsart}

\makeatletter \@addtoreset{equation}{section}

\makeatother \makeatletter
\newtheorem{theorem}{Theorem}[section]
\newtheorem{corollary}[theorem]{Corollary}
\newtheorem*{theorem*}{Theorem}
\newtheorem{definition}[theorem]{Definition}
\newtheorem{lemma}[theorem]{Lemma}
\newtheorem{proposition}[theorem]{Proposition}
\newtheorem*{Proposition}{Proposition}
\newtheorem{remark}[theorem]{Remark}
\newtheorem{example}[theorem]{Example}

\def\pr1{\prod\hskip -2.07ex * \hskip 0.9 ex}

\usepackage{latexsym} \usepackage{amsmath} \usepackage{amssymb}
\usepackage{amsthm}

\usepackage{cancel}\usepackage{color}
\begin{document}
\noindent
\title{Resultants of slice regular polynomials in two quaternionic variables}
\author{Anna Gori $^1$} 
 \thanks{$^1$ Dipartimento di Matematica - Universit\`a di Milano,                
 Via Saldini 50, 20133  Milano, Italy ({anna.gori@unimi.it})}
 \author{Giulia Sarfatti $^2$}
 \thanks{ $^2$ DIISM - Universit\`a Politecnica delle Marche,               
 	Via Brecce Bianche 12,  60131, Ancona, Italy ({g.sarfatti@univpm.it})}
 \author{Fabio Vlacci $^3$}
 \thanks{$^3$ DiSPeS  - Universit\`a di Trieste, Piazzale Europa 1, 34100, Trieste, Italy ({fvlacci@units.it})}

 \begin{abstract} We  introduce a non-commutative
   resultant,  for slice regular
   polynomials in two quaternionic variables, defined in terms of a suitable Dieudonn\'e determinant.
   We use this tool to investigate the existence of
   common zeros of slice regular polynomials.

 \end{abstract}
\keywords{Quaternionic slice regular polynomials, vanishing sets, non-commutative resultant\\
{\bf MSC:} 30G35, 16S36, 15A54} 
\maketitle
\noindent
\section{Introduction}    

{ The description of vanishing loci of polynomials in relation with their factorization, together with the study of polynomial systems, is relevant both in the commutative and in the non-commutative settings; it has not only an intrinsic value, but also natural applications in several fields ranging from pure mathematics, such as algebraic geometry (see, e.g., \cite{CLO}), to advanced applied sciences, such as mechanical sciences, robotics, computer animation or geometric design (see, e.g., \cite{FGGSS, GST, lercher}). 
}

    The study of common zeros and common irreducible factors of polynomials in the complex setting, is naturally related with the notion of {\em
  resultant}. In the non-commutative setting  first results
in this direction have been obtained, e.g., in \cite{ cileni, erik,
  ZZ}  for polynomials in one  quaternionic variable. In \cite{Rasheed}, a notion of resultant in the setting of {\em Ore polynomial rings} in two variables has been defined.
  
In this paper we consider the class $\mathbb H[q_1,q_2]$ of {\em slice regular} polynomials of
two quaternionic variables (see Definition \ref{poly}) and we
introduce a {\em regular resultant} for $P,Q \in\mathbb H[q_1,q_2]$, defined in terms of a (suitable)
determinant, namely an adaptation of the Dieudonn\'e determinant over
the quaternions (see, e.g., \cite{Ash, bren, Dieudonne}). 

Although our definition is formally similar to the one given by Rasheed in \cite{Rasheed}, 
he defines the resultant only  with respect to one of the two variables, while with our approach
we can define two resultants ${\rm Res}(P,Q;q_1)$ and ${\rm Res}(P,Q;q_2)$ with respect to $q_1$ or to $q_2$; these resultants  turn out to be slice regular polynomials  in one quaternionic variable and therefore they belong to $\mathbb{H}[q_2]$  and to $\mathbb{H}[q_1]$ respectively.
In \cite{Rasheed} the author is interested in finding more efficient ways to compute the resultant, while our major aim is to apply the theory of resultants to investigate the existence of common zeros and common irreducible factors of quaternionic slice regular polynomials. 
In particular we prove
\begin{theorem*}
Given two slice regular polynomials $P,Q$ in two quaternionic variables, both their {\em regular resultants} belong to the right ideal generated by $P$ and $Q$.
\end{theorem*}  
\noindent It is possible to prove that all common zeros  $(a,b)$ with $ab=ba$
of two slice regular  polynomials $P,Q$ in two quaternionic variables
sit in the zero locus of their regular resultants. More precisely
\begin{Proposition}
Let  $P$ and $Q$ be  slice regular polynomials in two quaternionic variables. If they have a common zero $(a,b)\in \mathbb{H}^2$ such that $ab=ba$,
then ${\rm Res}(P,Q;q_1)$ vanishes at $q_2=b$ and ${\rm Res}(P,Q;q_2)$ vanishes at $q_1=a$.
\end{Proposition}	

A first result relating the resultant with the existence of common irreducible factors is the following 
\begin{Proposition}
  Let $P$ and $Q$ be  slice regular polynomials in two quaternionic variables and let $a\in \mathbb{H}$.
	\begin{enumerate}
		\item If $P$ and $Q$ have the common  left factor $(q_1-a)
		$, then
		${\rm Res}(P,Q;q_1)\equiv 0$;
		\item if $P$ and $Q$ have the common  left  factor $(q_2-a)$,  then
		${\rm Res}(P,Q;q_2)\equiv 0$.
		\end{enumerate}
\end{Proposition}
In the other direction, we obtain the following result, stated in terms of the {\em symmetrizations} of the involved slice regular polynomials (see Definition \ref{R-coniugata2}).
\begin{Proposition}
	Let $P$ and $Q$ be slice regular polynomials in two quaternionic variables. If  \ ${\rm Res}(P,Q;q_1)\equiv 0$ or ${\rm Res}(P,Q;q_2)\equiv 0$, then the symmetrized polynomials $P^s$ and $Q^s$ have a common factor.
\end{Proposition}

The present paper is organized as follows:
in Section \ref{sec1} we shortly recall the main definitions and results from the theory of
slice regular and semi-regular functions which will be
used in the sequel.
In Section \ref{sceD} we introduce a suitable version of Dieudonn\'e determinant for
square matrices with entries in a skew-field, which we specialize to the case of semi-regular functions. In Section \ref{RES} we apply
this determinant to define a notion of regular resultant for slice
regular polynomials and we use this new notion to investigate the existence of common zeros and common irreducible factors.
Finally, in analogy with the complex case, we apply regular resultants to study multiplicities of zeros of slice regular polynomials.

\section{Introduction to quaternionic slice regular functions}\label{sec1}
\noindent
Let $\mathbb{H}{=\mathbb{R}+i\mathbb{R}+j\mathbb{R}+k\mathbb{R}}$ denote the skew field of
quaternions and let $\mathbb{S}=\{q \in \mathbb{H} \ : \ q^2=-1\}$ be
the two dimensional sphere of quaternionic imaginary units.  Then
\[ \mathbb{H}=\bigcup_{I\in \mathbb{S}}(\mathbb{R}+\mathbb{R} I),  
\]
where, for any $I\in \mathbb{S}$, the ``slice'' $\mathbb{C}_I:=\mathbb{R}+\mathbb{R} I$ can be
identified with the complex plane $\mathbb{C}$.  { Hence, any $q\in \mathbb{H}$ can be
  expressed as $q=x+yI$ with $x,y \in \mathbb{R}$ and $I \in
  \mathbb{S}$. The {\it real part }of $q$ is ${\rm Re}(q)=x$ and its
         {\it imaginary part} is ${\rm Im}(q)=yI$; the {\it conjugate}
         of $q$ is $\bar q:={\rm Re}(q)-{\rm Im}(q)$.  For any non-real
         quaternion $a\in \mathbb{H}\setminus \mathbb{R}$ we will
         denote by $I_a:=\frac{{\rm Im}(a)}{|{\rm Im}(a)|}\in
         \mathbb{S}$ and by $\mathbb{S}_a:=\{{\rm Re}(a)+I |{\rm
           Im}(a)| \ : \ I\in \mathbb{S}\}$. If $a\in \mathbb{R}$,
         then $I_a$ is any imaginary unit.}\\ Let us now recall the
definition of slice regular functions of a quaternionic variable.  In
view of the setting we will be working on, we will restrict to the
simplest case possible, that is to slice regular functions defined on
open balls centered at the origin, $B_R:=\{q \in \mathbb{H} \ :
\ |q|<R \}$, with radius $R>0$, possibly infinite.

For all details and a more general version of this theory  
we refer the reader to
\cite{libroGSS}. 

\begin{definition}
Let $R>0$. A function $f: B_R \to \mathbb{H}$ is called
\emph{slice regular} if the restriction $f_I$ of $f$ to $B_R\cap \mathbb{C}_I$
is holomorphic {for any $I\in\mathbb{S}$}, i.e., if $f_I$ has continuous partial derivatives and it is such that
\[\overline{\partial}_If_I(x+yI):=\frac{1}{2}\left(\frac{\partial}{\partial x}+I\frac{\partial}{\partial y}\right)f_I(x+yI)=0\]
for all $x+yI\in B_R 
\cap \mathbb C_I$ {and for any $I\in\mathbb{S}$}.
\end{definition}
In the sequel we will denote by $\mathcal{SR}(B_R)$ the class of
slice regular functions on $B_R \subseteq \mathbb{H}$.
{
A class of slice regular functions, crucial for our study,
is the set of {\em slice regular polynomial functions}, which we denote by
\[\mathbb{H}[q]:=\left\{P(q)=\sum_{n=0}^{N}q^na_n : \deg P:=N\in\mathbb{N},  a_n \in \mathbb{H} \  
\text{for any $n=0,\ldots, N$}\right\}.\]
More in general, the class of {\em slice regular power
  series} of the form $\sum\limits_{n\in\mathbb N}q^na_n$ with $a_n\in\mathbb{H}$, convergent in $B_R$, 
coincides with the class of regular function in
$B_R$.

\begin{definition} Let $R>0$ and
  let $f:B_R \to \mathbb{H}$
be a slice regular function. For any $I\in\mathbb{S}$ and
any point $q=x+yI$ in $B_R$ 
we define the  {\em Cullen derivative} of $f$ at $q$ as
$$\partial_C f(x+yI):=\dfrac{1}{2}\left(\frac{\partial}{\partial x}
-I\frac{\partial}{\partial y}\right)f_I(x+yI)$$
\end{definition}

}
While the sum of two slice regular
functions is clearly slice regular, an appropriate notion of
multiplication in $\mathcal{SR}(B_R)$ is given by the so called {\em $*$-product}.
\begin{definition}
	If $f(q)=\sum\limits_{n\in\mathbb N} q^n a_n$ and
        $g(q)=\sum\limits_{n\in\mathbb N} q^n b_n$ are two slice regular functions in $B_R$, then the $*$-product of $f$ and $g$ is the slice regular
        function in $B_R$ defined by
	$$
	f\ast g(q):=\sum_{n\in\mathbb N}q^n\sum_{k=0}^na_k b_{n-k}.
	$$
	\end{definition}
\noindent Notice that the $*$-product is associative but not commutative in general.
Therefore $(\mathcal{SR}(B_R), +, *)$ is a non-commutative ring.  
Furthermore the Leibniz rule holds for $*$-product of slice regular functions with respect to the Cullen derivative.

\noindent It is not difficult to see that if $f(q)=\sum_{n \in
  \mathbb{N}}q^na_n \in \mathcal{SR}(B_R)$ has real coefficients (i.e. $f$ is {\em
  slice preserving}), then, for any slice regular function $g \in \mathcal{SR}(B_R)$,
$
f*g(q)=
f(q)\cdot g(q)=g*f(q),
$
on $B_R$.  

\noindent In general, the relation of the $*$-product with the usual pointwise product is the following:
	$$
	f* g(q)=\left\{\begin{array}{c cl}
	0 & &\text{if $f(q)=0 $}\\
	f(q)\cdot g(f(q)^{-1}\cdot q \cdot f(q)) & &\text{if $f(q)\neq 0 $\,,}
	\end{array}\right.
	$$
	for any $f,g \in \mathcal{SR}(B_R,+,*)$.
	\noindent
        Notice that $f(q)^{-1}\cdot q\cdot f(q)$ belongs to the same
        sphere $x+y\mathbb{S}$ as $q$. Hence each zero of $f*g$ in
        $x+y\mathbb{S}$ is given either by a zero of $f$ or by a point
        which is a conjugate of a zero of $g$ in the same sphere.
	
A peculiar aspect of slice regular functions is the structure of their zero
sets. In fact, besides isolated zeros, they can also vanish on two
dimensional spheres. As an example, the slice regular polynomial $q^2+1$ vanishes on
the entire sphere of imaginary units $\mathbb{S}$.  It can be proven
that also spherical zeros cannot accumulate.
\begin{theorem}
	Let $f$ be a slice regular function on $B_R$. If $f$ does not vanish identically, then its zero
        set consists of isolated points or isolated $2$-spheres of the
        form $x +y \mathbb{S}$ with $x,y \in \mathbb{R}$, $y \neq 0$.
\end{theorem}
Furthermore it holds
\begin{theorem}[Identity Principle]\label{Id}
	Let $f, g$ be two slice regular functions in $B_R$. If there exists
	$I \in \mathbb S$ such that $f\equiv g$ on
	$B_R\cap
	\mathbb{C}_{I}$, then
	$f\equiv g$ on $B_R$.
	\end{theorem}

In order to define a {\em regular inverse}, we need to introduce the following operations.
\begin{definition}\label{R-coniugata}
	Let $f(q)=\sum_{n \in \mathbb{N}}q^na_n$ be a slice regular function
        on $B_R$. Then the \emph{regular conjugate} of $f$ is the
        slice regular function defined on $B_R$ by $f^c(q)=\sum_{n \in
          \mathbb{N}}q^n\overline{a_n}$, and the \emph{symmetrization}
        of $f$ is the slice regular function defined on $B_R$ by $f^s =
        f*f^c = f^c*f$.
\end{definition}
\noindent The \emph{regular inverse} $f^{-*}$ of a slice regular function $f:B_R \to \mathbb{H}$ is defined as
\begin{equation}
f^{-*}=(f^{s})^{-1}f^c,
\end{equation}
and it is slice regular on $B_R\setminus Z_{f^s}$, where $Z_{f^s}$ denotes the zero set of the symmetrization $f^s$.
The function  $f^{-*}$ is a semi--regular function in the sense of \cite{singularities}.
{

\noindent We now recall the definition of Ore domains, which will be used in the sequel (see, e.g., \cite{singularities}).
\begin{definition}
	A {\em left (right) Ore domain} is a domain $(D,+,\star)$, that is a
        ring without zero divisors, such that $D\star a \cap D\star b
        \neq\{0\}$ (respectively, $a\star D \cap b\star D \neq \{0\}$), 
        for any $a,b \in D\setminus\{0\}$, where
  $D\star a:=\{d\star a  \ : \ d \in D\}$ (and $a \star D:=\{a\star d  \ : \ d \in D\}$).
	\end{definition}

}

\begin{theorem} Let $(D,+,\star)$ be a left Ore domain. Then
the set of formal quotients 
$ L = \{a^{-\star}\star b \ : \ a,b\in D\}$ can be endowed with operations $+,\star$ such that:
\begin{enumerate}
	\item $D$ is isomorphic to the subring $\{1^{-\star}\star a \ : \ a\in D\}$ of $ L$;
\item	$ L$ is a skew field, i.e. a ring where every non-zero element $a^{-\star}\star b$ has a multiplicative
inverse $(a^{-\star}\star b)^{-\star}=b^{-\star}\star a$.
\end{enumerate}
The ring $ L$ is called { \em the classical left ring of quotients of } $D$  and, up to isomorphism,
it is the only ring having the properties {\em (1)} and {\em (2)}.
\end{theorem}

\noindent On a right Ore domain $D$, the
classical right ring of quotients is similarly constructed. If $D$ is both a left and a
right Ore domain, then (by the uniqueness property) the two rings of quotients are
isomorphic and we speak of the classical ring of quotients of $D$, 
(see \cite{Cohn} for the general theory of ring of quotients in the non-commutative setting).  

In particular, for slice regular functions, we have the following result (see \cite[Theorem 4.5]{singularities}).

\begin{theorem}
  The set of left $*$-quotients
\[\mathcal{L}(B_R)=\{f^{-*}*g \ : f,g \text{ slice regular in } B_R, f\not\equiv 0  \}\]
is a division ring with respect to $+, *$.  Moreover, the ring
$\mathcal{SR}(B_R)$ of slice regular functions on $B_R$ is a left and
right Ore domain and $\mathcal{L}(B_R)$ is its classical ring of
quotients.
\end{theorem}

\noindent All functions in $\mathcal{L}(B_R)$ are semi--regular functions on $B_R$. \\ 
The same construction applies to the ring  $(\mathbb{H}[q], +, *)$, 
leading to the definition of the skew-field of quotients of
slice regular polynomials.

{
}

\subsection{Several variable case}
When considering functions of several quaternionic variables, the
definition of slice regularity is more complicated and relies on
the notion of {\em stem functions}.
{
The formulation of this theory (in
the more general setting of real alternative $*-$algebras)
can be found in \cite{severalvariables}. For our purposes, we will only deal with slice regular polynomials and convergent slice regular power series
of two 
quaternionic variables. All the notions and properties discussed in this section can be easily adapted to the several variables case.}
\begin{definition}\label{poly}
	A slice regular polynomial $P: \mathbb{H}^2\to \mathbb{H}$ is any function of the form 
	\[P(q_1,q_2)=\sum_{ \substack{n=0,\ldots,N  \\ m=0,\ldots,M} }{q_1}^n{q_2}^ma_{n,m} \]
	with $a_{n,m}\in\mathbb{H}$. We set $\deg_{q_1}P:=N$ and $\deg_{q_2}P:=M$.\\
 More generally, for $R_1,R_2>0$, a power series of the form
	\[\sum_{(n,m)\in \mathbb{N}^2}q_1^nq_2^ma_{n,m},\] converging in $B_{R_1}\times B_{R_2}=:\Omega_{R_1,R_2}$, defines a slice regular function $f(q_1,q_2)$ on $\Omega_{R_1,R_2}$. 
\end{definition}

{For these classes of slice regular functions it is possible to define in a natural way {\em partial Cullen derivatives} (see  \cite{severalvariables, PV}). 
\begin{definition}
Let $f(q_1,q_2)=\sum_{(n,m)\in \mathbb{N}^2}q_1^nq_2^ma_{n,m}$ be a slice regular function on $\Omega_{R_1,R_2}$. Then 
the partial Cullen derivative of $f$ with respect to $q_1$ is 
	\[f_{q_1}(q_1,q_2)=\sum_{n\ge 1,m \ge 0}n q_1^{n-1}q_2^ma_{n,m},\]
the partial Cullen derivative of $f$ with respect to $q_2$ is 
\[f_{q_2}(q_1,q_2)=\sum_{n\ge 0,m \ge 1}m q_1^{n}q_2^{m-1}a_{n,m}.\]
\end{definition}

}

Slice regular functions of several variables can be endowed with an
appropriate notion of product, the so called {\em slice product}. It
will be denoted by the same symbol $*$ used in the one-variable
case. We refer to \cite{severalvariables} for a precise definition,
let us recall here how it works for
convergent power series.
\begin{definition}
	If $f(q_1,q_2)=\sum\limits_{n,m\in\mathbb N} q_1^nq_2^m a_{n,m}$ and
        $g(q)=\sum\limits_{n,m\in\mathbb N} q_1^nq_2^m b_{n,m}$ are two
        slice regular functions in  
        $\Omega_{R_1,R_2}$,
        then the {\em $*$-product} of $f$ and $g$ is the slice regular function in $\Omega_{R_1,R_2}$, defined by
$$
f*g (q_1,q_2):=\sum_{n,m\in\mathbb N}q_1^nq_2^m\sum_{ \substack{r=0,\ldots,n  \\ s=0,\ldots,m} }a_{r,s} b_{n-r,m-s}
$$
\end{definition}
\noindent For example, if $a,b \in \mathbb{H}$, then
\begin{itemize}
	\item $q_1*q_2=q_2*q_1=q_1q_2$;
	\item $a*(q_1q_2)=(q_1q_2)*a=q_1q_2a$;
	\item $(q_1^nq_2^ma)*(q_1^rq_2^sb)=q_1^{n+r}q_2^{m+s}ab$.
\end{itemize} 	
\noindent Moreover we point out that, if $f$ or $g$ have real coefficients, then $f*g=g*f$.

{\begin{remark}
As proven in \cite{severalvariables}, partial Cullen derivatives satisfy the Leibniz rule with respect to the $*$-product.	
\end{remark}
}

\medskip 

\noindent In the sequel we will denote by $\mathbb{H}[q_1,q_2]$ the set of slice regular polynomials in $2$ quaternionic variables. Since the $*$-product is associative but not commutative, $(\mathbb{H}[q_1,q_2], +, *)$ is
a non-commutative ring (without zero divisors). 

{
	\begin{remark}\label{valutazione}
		Point-wise  evaluation of slice regular functions
		is not in general a multiplicative ring homomorphism.
		Therefore, the zeros of the $*$ product of two slice regular polynomials is not necessarily given by the union
		of the zero sets of each of the factors.
		For instance, $q_1-i$ vanishes on $\{i\}\times \mathbb{H}$,
		while $(q_1-i)*(q_2-j)=q_1q_2-q_1j-q_2i+k$, when $q_1=i$, vanishes only for  $q_2\in\mathbb{C}_i$.
		
	\end{remark}
}



The ring $\mathbb{H}[q_1,q_2]$
of slice regular polynomials in two quaternionic variables 
can be also regarded as the ``regular extension'' of the ring $\mathbb{H}[q_2]$ adding the variable $q_1$. In fact
	let $P\in \mathbb{H}[q_1,q_2]$ be a slice regular polynomial in two quaternionic variables,
then $P$ can be written as a polynomial in the variable $q_1$ with coefficients in $\mathbb{H}[q_2]$:
	
	\begin{equation}\label{coefficienti1}
		P(q_1,q_2)=\sum_{n=0}^{\deg_{q_1} P} q_1^n *P_n(q_2)=\sum_{n=0}^{\deg_{q_1} P} q_1^n P_n(q_2)
	\end{equation}
		where each $P_n$ is a regular polynomial in the
        quaternionic variable $q_2$.  We will use the notation 
        $[\mathbb{H}[q_2]]_{\mathcal{SR}}[q_1]$ instead of $\mathbb H[q_1,q_2]$ when needed. 
        Moreover, if one considers
        another slice regular polynomial
        $Q(q_1,q_2)=\sum\limits_{s=0}^{\deg_{q_1} Q} q_1^s Q_s(q_2)$
        in $[\mathbb{H}[q_2]]_{\mathcal{SR}}[q_1]$, then the $*$-product of $P$ and
        $Q$ is given by
	\[P*Q (q_1,q_2) =\sum_{k=0}^{\deg_{q_1}P+\deg_{q_1}Q} q_1^k C_k(q_2)\]
	where the coefficients are $C_k=\sum\limits_{n+s=k} P_n * Q_s$.

Exchanging the role of the two variables, one gets something slightly different. 
In fact, every slice regular polynomial $P\in \mathbb{H}[q_1,q_2]$ can be also written as
\begin{equation}\label{coefficienti3}
P(q_1,q_2)=\sum_{m=0}^{\deg_{q_2} P} q_2^m *\tilde{P}_m(q_1)
\end{equation}
where $\tilde P_m$ is a slice regular polynomial in the quaternionic
variable $q_1$, for any $m$.
We point out that in \eqref{coefficienti3}
there is no immediate analogue formulation in terms of pointwise product as in \eqref{coefficienti1}.
We will denote by
$[\mathbb{H}[q_1]]_{\mathcal{SR}}[q_2]$ the set of slice regular polynomials
in the variable $q_2$ with coefficients in $\mathbb{H}[q_1]$, taking
into account that the multiplication involved is the $*$-product.
If $Q$ is another slice regular polynomial in $[\mathbb{H}[q_1]]_{\mathcal{SR}}[q_2]$,
\[Q(q_1,q_2)=\sum_{r=0}^{\deg_{q_2} Q} q_2^r *\tilde Q_r(q_1),\]
 then the $*$-product of $P$ and $Q$ can also be expressed as
	\[P*Q (q_1,q_2) =\sum_{k=0}^{\deg_{q_2}P+\deg_{q_2}Q} q_2^k* \tilde C_k(q_1)\]
where the coefficients are $\tilde C_k=\sum\limits_{n+s=k} \tilde P_n * \tilde Q_s$.


	\begin{remark}
	The same can be done in $n$ variables: let $P$ be a slice regular polynomial in $n$ variables, namely 
\[P(q_1, \ldots, q_n)=\sum_{(k_1,\ldots, k_n)\in K}
q_1^{k_1}q_2^{k_2}\cdots q_n^{k_n} a_{k_1,\ldots, k_n}\]
where $K$ is a finite subset of $\mathbb{N}^n$. If $q_l$ is one of the variables of $P$, then
one can write
\[P(q_1,\ldots,q_n)=\sum_{ k=0 }^{\deg_{q_l}P} q_l^{k}*\tilde P_k(q_1, \ldots,q_{l-1},q_{l+1}, \ldots, q_n),\]
where $\tilde P_k(q_1, \ldots,q_{l-1},q_{l+1}, \ldots, q_n)$ is a slice regular
is a slice regular polynomial in the remaining variables.

\noindent Hence the set of slice regular polynomials $P$ in $n$ variables $\mathbb{H}[q_1, \ldots,q_{n}]$ 
can be regarded as the {``regular extension''}
$[\mathbb{H}[q_1, \ldots,q_{l-1},q_{l+1}, \ldots,
    q_n]]_{\mathcal{SR}}[q_l]$,
for any $l=1,\ldots, n$.

\end{remark}

	As in the one-variable case, it is possible to introduce two
        operators on slice regular functions (see \cite{severalvariables}).
\begin{definition}\label{R-coniugata2}
	Let $f(q_1,q_2)=\sum\limits_{n,m \in \mathbb{N}}q_1^nq_2^ma_{n,m}$ be
        a slice regular function on the domain $\Omega_{R_1,R_2}$. Then the \emph{regular conjugate} of $f$ is the
        slice regular function defined on $\Omega_{R_1,R_2}$ by
	\[f^c(q_1,q_2)=\sum_{n,m \in \mathbb{N}}q_1^nq_2^m\overline{a_{n,m}},\]
	and the
	\emph{symmetrization} of $f$ is the slice regular function defined on $\Omega_{R_1,R_2}$ by 
	\[f^s= f*f^c = f^c*f.\]
\end{definition}
\noindent The behavior of these operators with respect to the $*$-product is the same as the one they have in the one-variable case. 
\begin{remark}\label{sym} 
By direct computation, we have that the coefficients of the power series expansion of $f^s$ are real numbers and hence 
$f^s*g=g*f^s$ for any slice regular function $g$ (where defined).
Moreover, for any $f,g$ slice regular functions, we also have that $(f*g)^c=g^c*f^c$ and hence that $(f*g)^s=f^s*g^s$ (where defined).
\end{remark}
\begin{remark}\label{ZD}
Notice that, since $(\mathbb{H}[q_1, q_2],+,*)$ does not contain zero divisors, if $P$ is a slice regular polynomial not identically zero, then its symmetrization $P^s\not \equiv 0.$ 
\end{remark}

 The next proposition (see \cite{GSV-LAA} {for its formulation in $n$-variables}) gives a first geometrical  description of the zero set of  slice regular polynomials in  two quaternionic variables.
 
 \noindent For any $a \in \mathbb H$, let us denote by 
 \[C_a=\{q \in \mathbb{H} \ : \ qa=aq\}.\]

	\begin{proposition}\label{mlineare}
	Let $P\in \mathbb{H}[q_1, q_2]$
	be a slice regular polynomial in two variables and let $a \in \mathbb H$. Then $P$ vanishes on $\{a\}\times C_a$
	if and only if there exists $P_1\in\mathbb{H}[q_1, q_2]$ such that 
	\[P(q_1,q_2)=(q_1-a)*P_1(q_1, q_2).\]
	The slice regular polynomial $P$ vanishes on $\mathbb H \times \{a\}$ if and only if there exists $P_2\in\mathbb{H}[q_1, q_2]$ such that 
	\[P(q_1,q_2)=(q_2-a)*P_2(q_1, q_2).\]
\end{proposition}

Moreover (see again \cite{GSV-LAA}) we will make use of the following

        \begin{proposition}\label{zeri*}
  If a slice
        regular polynomial $P$ vanishes at $(a,b)$ with
        $ab=ba$, then $P*Q$ still
        vanishes at $(a,b)$ for any $Q$ in $\mathbb
        H[q_1,q_2]$. 

          \end{proposition}

	\section{Preliminaries on non commutative  determinants}\label{sceD}

Consider a skew field $(\mathbb{F}, +,\star)$, where the product operation $\star$ is  not necessarily  commutative.
	Let $\overline{\mathbb{F}}^{\star}$
	be the maximal abelian quotient of the multiplicative group  $(\mathbb{F},\star)$, i.e. the quotient of  $(\mathbb{F},\star)$
	by its commutator subgroup. In general we will use the notation
        $[x]_{\overline{\mathbb{F}}^{\star}}$ to consider the image in
        $\overline{\mathbb{F}}^{\star}$ of $x\in\mathbb{F}$.  If
        $[x]_{\overline{\mathbb{F}}^{\star}}$ and
        $[y]_{\overline{\mathbb{F}}^{\star}}$ are given, then, using the same symbol $\star$ to denote the multiplication operator in $\overline{\mathbb{F}}^{\star}$, we will write
        $[x]_{\overline{\mathbb{F}}^{\star}}\star
        [y]_{\overline{\mathbb{F}}^{\star}}:= [x\star
          y]_{\overline{\mathbb{F}}^{\star}}$.  Notice that, by
        definition, $[x]_{\overline{\mathbb{F}}^{\star}}\star
        [y]_{\overline{\mathbb{F}}^{\star}}=
        [y]_{\overline{\mathbb{F}}^{\star}}\star
        [x]_{\overline{\mathbb{F}}^{\star}}$.
	
	The multiplicative group of invertible $n\times n$ matrices with entries in $\mathbb{F}$, will be denoted
	by $GL(n, \mathbb{F})$,
where, if $A=(a_{ij})_{1\leq
			i,j\leq n}$ and $B=(b_{ij})_{1\leq i,j\leq n}$, with $a_{ij}\in\mathbb{F}$ and $b_{ij}\in\mathbb{F}$ for $1\leq i,j\leq n$,
		then $A\star_{GL(n,\mathbb{F})} B:=(\sum\limits_{k=0}^n a_{ik}\star
		b_{kj})_{1\leq i,j\leq n}$.
		For the sake of shortness, we will simply use the symbol $\star$ for $\star_{GL(n,\mathbb{F})}$.
		The neutral element for $\star$ in $GL(n, \mathbb{F})$ 
		is $Id_n=(\delta_{ij})_{1\leq i,j\leq n}$, where $\delta_{ij}$ is the Kronecker symbol.
		Clearly $A\star B\neq B\star A$ in general.

	We recall the following (see, e.g., \cite{Dieudonne, Ash, Cas})
	
	\begin{proposition}\label{Pdet}
		There exists a unique homomorphism
		\[ { \rm Det}^{\star}_n: GL(n,\mathbb{F})\to \overline{\mathbb{F}}^{\star}\]
		such that
		\[
		{\rm Det}^{\star}_n\begin{pmatrix}
		\lambda_1& 0&\ldots&\dots&0\\
		0&\lambda_2&0&\ldots&0 \\
		\vdots& \vdots &\ddots&0&\vdots\\
		\vdots& \vdots &\vdots&\ddots&0\\

		0&\ldots&0&0&\lambda_n
		\end{pmatrix}=[\lambda_1\star \lambda_2\star \cdots \star \lambda_n]_{\overline{\mathbb{F}}^{\star}}.
		\]
			\end{proposition}

\noindent In the commutative setting ${\rm Det}^{\star}_n$ coincides with the usual determinant. In the noncommutative case, it defines a Dieudonn\'e type determinant, whose principal properties are:
	\begin{itemize}
	\item [(a)]${\rm Det}^{\star}_n(Id_n)=[1]_{\overline{\mathbb{F}}^{\star}}$ for any $n\geq 2$;
	\item [(b)] if $\mathbf{r}$ and $\mathbf{s}$ are distinct rows of $A\in GL(n,\mathbb{F})$
	and if $\tilde A$ is the matrix obtained from $A$ by replacing the row $\mathbf{r}$ with the row\footnote{With
		the following convention : if $\mathbf{s}=(a_{s1}\ a_{s2}\ \ldots \ a_{sn})$ and $\mathbf{r}=(a_{r1}\ a_{r2}\ \ldots \ a_{rn})$, then
		$\mathbf{r}+\lambda\star\mathbf{s}:=(a_{r1}+\lambda\star a_{s1}\ a_{r2}+\lambda\star a_{s2}\ \ldots \ a_{rn}+\lambda\star a_{sn})$.}
	$\mathbf{r}+\lambda\star \mathbf{s}$,
	then   ${\rm Det}^{\star}_n(\tilde A)={\rm Det}^{\star}_n(A)$;
	\item[(c)] if $\mathbf{r}$ is a row of $A\in GL(n,\mathbb{F})$
	and if $\hat A$ is the matrix obtained from $A$ by replacing the row $\mathbf{r}$ with the row
	$\lambda\star \mathbf{r}$,
	then   ${\rm Det}^{\star}_n(\hat A)=[\lambda]_{\overline{\mathbb{F}}^{\star}}  \star {\rm Det}^{\star}_n(A)$.
	\end{itemize}
        
\noindent As customary, one then extends the definition of ${\rm Det}^{\star}_n$ also to non-invertible square matrices 
and put  ${\rm Det}^{\star}_n(A)=[0]_{\overline{\mathbb{F}}^{\star}}$ if $A$ is a non-invertible square matrix, where $0\in\mathbb{F}$ is the neutral element for $+$.
	
	
	

\noindent The homomorphism ${\rm Det}^{\star}_n$ can be defined recursively, by means of properties (a), (b) and (c).	
{For the ease of the reader, we enclose here the main properties of ${\rm Det}^{\star}_n$, together with some of their proofs adapted to our setting.}

\noindent For $n=2$, if $a\neq 0$ 
	and if $a^{-\star}$ is the inverse of $a$ with respect to $\star$, then
	
	\[
	{\rm Det}^{\star}_2\begin{pmatrix}
	a&b\\
	c&d\\
	\end{pmatrix}:=[a\star d-a\star c\star a^{-\star}\star b]_ {\overline{\mathbb{F}}^{\star}}
	=[ c\star a\star c^{-\star}\star d  -c\star b]_ {\overline{\mathbb{F}}^{\star}};
	\]

        if $a=0$, then
	\[
	{\rm Det}^{\star}_2\begin{pmatrix}
	0&b\\
	c&d\\
	\end{pmatrix}:=[b\star c]_ {\overline{\mathbb{F}}^{\star}}.
	\]
	Clearly ${\rm Det}^{\star}_2 :GL(2,\mathbb{F})\to
        \overline{\mathbb{F}}^{\star}$ satisfies the condition for
        diagonal $2\times 2 $ matrices as in Proposition \ref{Pdet} and hence also the corresponding property (a) holds.
	Moreover, by direct computations, one can easily see that for any $\lambda \in \mathbb{F}$

	\begin{equation}\label{row}
	{\rm Det}^{\star}_2\begin{pmatrix}
	a+\lambda\star c&b+\lambda\star d\\
	c&d\\
	\end{pmatrix}=
        {\rm Det}^{\star}_2\begin{pmatrix}
a&b\\
c+        \lambda\star a&d+\lambda\star b\\
\end{pmatrix}=
         {\rm Det}^{\star}_2 \begin{pmatrix}
	a& b\\
	c&d\\
	\end{pmatrix}
	\end{equation}

and

	\begin{equation}\label{row}
	{\rm Det}^{\star}_2\begin{pmatrix}
	\lambda\star a&\lambda\star b\\
	c&d\\
	\end{pmatrix}=
        {\rm Det}^{\star}_2\begin{pmatrix}
a&b\\
        \lambda\star c&\lambda\star c\\
\end{pmatrix}=
        [\lambda]_{\overline{\mathbb{F}}^{\star}}\star {\rm Det}^{\star}_2 \begin{pmatrix}
	a& b\\
	c&d\\
	\end{pmatrix}
	\end{equation}
which show that properties (b) and (c) hold for ${\rm Det}^{\star}_2$.

        In order to  define ${\rm Det}^{\star}_n$ for $n\geq 2$,
        	one proceeds recursively.  
	Suppose $A$ is an $n\times n$ matrix
	whose entries are in $\mathbb{F}$.
	If all the elements of the first column of $A$ are zero then ${\rm Det}^{\star}_n(A)=0$.
	If not, choose a non--zero element $a_{k1}$ in the first column of $A$ and subtract a
	suitable multiple of the $k$-th row of $A$ from each ot the other rows of $A$
	in such a way that any other element of the first column of $A$ vanishes.
	Call this matrix $B_k$.
	Then consider the $(n-1)\times (n-1)$ matrix $B_{k1}$ obtained from $B_k$ by eliminating the first column and the $k$-th row, and
	define\footnote{
	This is precisely what has been done for ${\rm Det}^{\star}_2$, since, when $a\neq 0$,
	
	\[
	{\rm Det}^{\star}_2\begin{pmatrix}
	a&b\\
	c&d\\
	\end{pmatrix}
	=[a\star d-a\star c\star a^{-\star}\star b]_ {\overline{\mathbb{F}}^{\star}}=
	{\rm Det}^{\star}_2\begin{pmatrix}
	a&b\\
	0&d- c\star a^{-\star}\star b
	\end{pmatrix}.
	\]
	}
	
	\begin{equation}\label{Dn}
	{\rm Det}^{\star}_n(A):=[a_{k1}]_{\overline{\mathbb{F}}^{\star}}\star{\rm Det}^{\star}_{n-1} B_{k1}.
	\end{equation}
        Notice that, from the way it is defined, the properties (a), (b) and (c)
        proved for  ${\rm Det}^{\star}_2$ extends automatically to ${\rm Det}^{\star}_n$.

	
	 In order to prove that the definition of ${\rm Det}^{\star}_n$ in (\ref{Dn}) is well-posed, one has to
	show that it is independent from the choice of the non-zero element in the first column of $A$:
        indeed this follows from the above-mentioned properties of ${\rm Det}^{\star}_n$ (see \cite{Cas} for the details).
        Another important consequence of  Proposition \ref{Pdet}
	(see \cite{bren})
	is a generalization of the Binet formula, namely if
	$A,B\in GL(n,\mathbb{F})$ then
	
	\begin{equation}\label{binet}
	{\rm Det}^{\star}_n(A\star B)={\rm Det}^{\star}_n(A)\star {\rm Det}^{\star}_n(B).
	\end{equation}  
	This fact leads to the following version of the Cramer Formula. 
	\begin{proposition}\label{Cramer}
	Consider a	system of $n$ linear equations in $n$ variables 
		\begin{equation}\label{sys}
			 A\star \mathbf{x}=\mathbf{b}
			 \end{equation}
	where $A=(a_{ij})_{1\leq    i,j\leq n}$ is a $n\times n$ matrix with entries in $\mathbb{ F}$, $\mathbf{x}=(x_1,\ldots,x_n)^t$ is the column of the unknowns and 
	$\mathbf{b}=(b_1,\ldots,b_n)^t\in \mathbb{F}^n$  is the column of constant terms. 
		Then ${\rm Det}^{\star}_n A\neq [0]_{\overline{\mathbb{F}}^{\star}}$ {if and only if} the system has a unique solution in $(\overline{\mathbb{F}}^\star)^n$, given by
	\begin{equation}\label{cramer}
	\begin{aligned}
		[\hat x_j]_{\overline{\mathbb{F}}^{\star}}&=
		{\rm Det}^{\star}_n(A)^{-\star}\star {\rm Det}^{\star}_n (A_1,A_2,\ldots, A_{j-1}, \mathbf{b}, A_{j+1}, \ldots, A_n)\\
&={\rm Det}^{\star}_n (A_1,A_2,\ldots, A_{j-1}, \mathbf{b}, A_{j+1}, \ldots, A_n)\star{\rm Det}^{\star}_n(A)^{-\star},
\end{aligned}
	\end{equation} 
	for $j=1,\ldots,n$, where $A_\ell$ denotes the $\ell$-th column of $A$ and the last equality holds since $\overline{\mathbb{F}}^\star$ is abelian.
	
	\end{proposition}
	\begin{proof}
		If ${\rm Det}^{\star}_n(A)\neq [0]_{\overline{\mathbb{F}}^{\star}}$, {then it is possible to prove (\cite[Theorem 1.6]{bren})} that the unique solution, up to commutators, of \eqref{sys} is given by ${\bf \hat x}=A^{-\star}\star {\bf b}$ .
Moreover, we can write

	\begin{equation*}
		{\footnotesize\begin{aligned}
			 A\star \begin{pmatrix}
	1&\ldots& 0&\overbrace{\hat x_1}^{j-th}&0&\ldots &0\\
		\vdots&\ddots&0&\vdots&\vdots&\ldots &\vdots\\
	\vdots&\ldots&1&\vdots&\vdots&\ldots &\vdots\\
	\vdots&\ldots&0&\ddots&0&\dots &\vdots\\
	\vdots&\ldots&\vdots&\vdots&1&\dots &0\\
	\vdots&\ldots&\vdots&\vdots&\vdots&\ddots &0\\
	0&\ldots&0&\hat x_n&0&\ldots&1
\end{pmatrix}
		&=(A_1,\ldots, A_{j-1}, A\star \mathbf{\hat x}, A_{j+1}, \ldots, A_n)=(A_1,\ldots, A_{j-1}, \mathbf{b}, A_{j+1}, \ldots, A_n)
\end{aligned}}
\end{equation*}
	where $A_k$ represents the $k$-th column of $A$.
	Therefore, using Binet Formula \eqref{binet}, and the properties of ${\rm Det}^{\star}_n$, we have
		\[{\rm Det}^{\star}_n(A)\star [ {\hat x_j }]_{\overline{\mathbb{F}}^{\star}}={\rm Det}^{\star}_n (A_1,A_2,\ldots, A_{j-1}, \mathbf{b}, A_{j+1}, \ldots, A_n),\]
	which leads to the conclusion.
	
	{Suppose now that ${\rm Det}^{\star}_n(A)= [0]_{\overline{\mathbb{F}}^{\star}}$. Using property (b) of ${\rm Det}^{\star}_n$, we can perform a row reduction of $A$ obtaining an upper triangular matrix  $A'$ such that ${\rm Det}^{\star}_n(A')={\rm Det}^{\star}_n(A)=0$ and hence with at least a zero row. Consider the $n\times (n+1)$ matrix $(A|b)$ whose first $n$ columns are those of $A$ and the $(n+1)$-th is $b$. There are two possibilities: either the row reduction of $(A|b)$ has the same number of zero rows of $A'$, or it has less zero rows than $A'$. In the first case the system has infinite solutions, in the latter the system has no solution.}
	\end{proof}

\noindent Versions of the Cramer Formula in the noncommutative setting can be found in, e.g., \cite{bren, cramer2, cramer1}.

\noindent 	Assume $(\mathbb{F},+,\star)= (\mathcal{L}(B_R), +, *)$, 
	 is the skew field of semi--regular functions in $B_R$, equipped with the $*$--product.
The corresponding Dieudonn\'e--type determinant for
$n\times n$  matrices with entries in $\mathcal{L}(B_R)$ will be denoted by ${\rm Det}_n^*
: M(n, \mathcal{L}(B_R))\to {\overline{\mathcal{L}(B_R)}^{*}}$.
\begin{proposition}\label{Ra} 

\noindent If $A\in  M(n, \mathbb{H}[q])\subset M(n, \mathcal{L}(\mathbb{H})) $, i.e. if all the entries of $A$ are slice regular polynomials in one quaternionic variable,
then one can find a  unique slice regular  polynomial $P\in \mathbb H[q]$, up to commutators, 
such that ${\rm Det}_n^*(A)=[P]_{\overline{\mathcal{L}(\mathbb{H})}^{*}}$.

\end{proposition}  

\noindent The  previous result is  Lemma 11
 proved in \cite{Rasheed}
  in the setting of Ore domains, rephrased here for matrices with entries in $\mathbb H[q]$. 
%

\section{Non commutative Resultant}\label{RES}

\noindent The theory of resultant is a  powerful tool in Algebraic Geometry
for the investigation of common zeros of polynomials. In the
quaternionic setting there have been done some attempts to
introduce a resultant of polynomials with
quaternionic coefficients
see, for instance, \cite{cileni} and \cite{erik} for
polynomials in one  variable 
and \cite{Rasheed}
for polynomials in two variables.

\noindent In the present section we want to
give a definition of {\em regular resultant} in $\mathbb{H}[q_1,q_2]$ and
study some of its properties in relation with the existence of common zeros and common factors of slice regular polynomials.
\noindent
In the complex setting (see, e.g., \cite{CLO}) two polynomials
$S$ and $T$ have an irreducible common factor if and only if there exist
two polynomials $U$ and $V$
with prescribed degrees such that
\begin{equation}\label{UFD} S\cdot U+T\cdot V\equiv 0.\end{equation}
This is a consequence of the fact that the ring of complex polynomials is a {\em Unique Factorization Domain} (UFD).
This is not true for the ring of slice regular quaternionic polynomials. For instance, in $\mathbb{H}[q]$, the
slice regular polynomial $q^2+1$ can be decomposed as
$q^2+1=(q-J)*(q+J)$ for any $J \in \mathbb{S}$. 

\noindent Condition \eqref{UFD} can be written as a homogeneous linear system
whose unknowns are the coefficients (that might be polynomials in the several variables
case) of $U$ and $V$. The existence of a non-zero solution is
equivalent to the vanishing of the determinant of the associated
matrix, known as the {\em Sylvester matrix} of $S$ and $T$. This
determinant is usually called the {\em resultant} of $S$ and $T$.


Let us apply this procedure in the setting of slice regular polynomials. 
Given $P,Q \in \mathbb{H}[q_1,q_2]$, we look for two slice regular polynomials $H,K\in \mathbb{H}[q_1,q_2]$, of proper degrees, such that 
 \begin{equation}\label{zerocomune}
 P*H+Q*K\equiv 0.
 \end{equation}
 Due to the non-commutativity of the two variables, the construction
 of the corresponding linear system  depends
 whether we choose to express $P,Q,H$ and $K$ as polynomials in $q_1$
 with coefficients in $\mathbb{H}[q_2]$ or as polynomials in $q_2$
 with coefficients in $\mathbb{H}[q_1]$.  More precisely, if
\[ \deg_{q_1}P=n, \ \deg_{q_2}P=r, \ \deg_{q_1}Q=m, \ \deg_{q_2}Q=s,\]
then, recalling equation \eqref{coefficienti1}, 
and considering $P,Q,H$ and $K$ in $[\mathbb{H}[q_2]]_{\mathcal{SR}}[q_1]$,
we can write 
 $$P(q_1,q_2)=\sum_{k=0}^{n} q_1^k P_k(q_2),
  \quad \quad 
Q(q_1,q_2)=\sum_{k=0}^{m} q_1^k
Q_k(q_2);$$
as in the complex case,  we look for $H$ and $K$, satisfying equation \eqref{zerocomune}, such that $\deg_{q_1}H<\deg_{q_1}Q$,  $\deg_{q_1}K<\deg_{q_1}P$,
and hence of the form
$$H(q_1,q_2)=\sum_{p=0}^{m-1} q_1^p H_p(q_2),  \quad  \quad 
K(q_1,q_2)=\sum_{p=0}^{n-1} q_1^p K_p(q_2).$$
Then equation \eqref{zerocomune} can  be written as
  \begin{equation}
  	\sum_{k=0}^{n} q_1^k P_k(q_2)*\sum_{p=0}^{m-1} q_1^p H_p(q_2)+\sum_{k=0}^{m} q_1^k
Q_k(q_2)*\sum_{p=0}^{n-1} q_1^p K_p(q_2)\equiv 0\label{fattorecomune}\end{equation} 
By the Identity Principle of Polynomials, this equation leads to the following homogeneous linear system:
\begin{equation}\label{linearsystem}
\left\{
\begin{array}{lcl}
P_0*H_0 + Q_0*K_0  &\equiv&0\\
P_1*H_0+P_0*H_1+Q_1*K_0+Q_0*K_1&\equiv&0\\
\vdots &\vdots&\\
P_n*H_{m-1}+Q_m*K_{n-1}&\equiv&0
\end{array}
\right.
\end{equation}
with unknowns  
$H_0,\cdots, H_{m-1},K_0,\cdots,K_{n-1}$ 
in ${\mathbb{H}[q_2]}$ and
associated matrix, with entries in ${\mathbb{H}[q_2]}$, of dimension ${(m+n)}\times{(m+n)}$
 $$A(q_2)=\begin{pmatrix}
P_0 &      0 & \cdots & 0 & Q_0 & 0 & \cdots & 0\\
P_{1}&P_0&\cdots&0&Q_{1}&Q_0&\cdots&0\\
P_{2}&P_{1}&\ddots&\vdots&Q_{2}&Q_{1}&\ddots&\vdots\\
\vdots&\vdots&\ddots&\vdots&\vdots&\vdots&\ddots&Q_0&\\
P_n&\vdots&\ddots&P_0&\vdots&\vdots&\ddots&Q_1\\
0&P_n&\ddots&P_1&\vdots&\vdots&\ddots&\vdots\\
\vdots&0&\ddots&\vdots&Q_m&Q_{m-1}&\ddots&\vdots&\\
\vdots&\vdots&\ddots&\vdots&0&Q_m&\ddots&\vdots\\
\vdots&\vdots&\ddots&\vdots&\vdots&\vdots& \ddots &\vdots \\
0&0&\cdots&P_n&0&0&\dots&Q_m\\
\end{pmatrix}.$$
On the other hand, if we choose to express $P,Q$ as polynomials  in $q_2$ with coefficients in $\mathbb{H}[q_1]$ then, by  equation \eqref{coefficienti3},  we have
\[P(q_1,q_2)=\sum\limits_{k=0}^{r}{q_2}^k*{\tilde P}_k(q_1) \quad \text{and} \quad Q(q_1,q_2)=\sum\limits_{k=0}^{{s} }{q_2}^k*{\tilde Q}_k(q_1).\]
 In this case, we look for $\tilde H$ and $\tilde K$ in $[\mathbb{H}[q_1]]_{\mathcal{SR}}[q_2]$ such that $\deg_{q_2}\tilde H<\deg_{q_2}Q$ and $\deg_{q_2}\tilde K<\deg_{q_2}P$,
 i.e. 
 \[\tilde{H}(q_1,q_2)=\sum\limits_{p=0}^{s-1} q_2^p* \tilde{H}_p(q_1)
 \quad \text{and} \quad
 \tilde {K}(q_1,q_2)=\sum\limits_{p=0}^{r-1} q_2^p *\tilde{K}_p(q_1),\]
   that satisfy 
\begin{equation}\label{zerocomune2}
P*\tilde H+Q*\tilde K\equiv 0.
\end{equation} 
 The previous equation, using again the Identity Principle for Polynomials, is equivalent to a homogeneous linear system in the unknowns  $\tilde{H}_0,\cdots , \tilde{H}_{s-1},\tilde{K}_0,\cdots,\tilde{K}_{r-1}$ in  ${\mathbb{H}[q_1]}$,
 whose associated matrix of dimension ${(r+s)}\times{(r+s)}$, with entries in  $\mathbb{H}[q_1]$, 
 is 
$$B(q_1)=
\begin{pmatrix}
\tilde P_0 &      0 & \cdots & 0 &\tilde Q_0 & 0 & \cdots & 0\\
\tilde P_{1}&\tilde P_0&\cdots&0&\tilde Q_{1}&\tilde Q_0&\cdots&0\\
\tilde P_{2}&\tilde P_{1}&\ddots&\vdots&\tilde Q_{2}&\tilde Q_{1}&\ddots&\vdots\\
\vdots&\vdots&\ddots&\vdots&\vdots&\vdots&\ddots&\tilde Q_0&\\
\tilde P_s&\vdots&\ddots&\tilde P_0&\vdots&\vdots&\ddots&\tilde Q_1\\
0&\tilde P_s&\ddots&\tilde P_1&\vdots&\vdots&\ddots&\vdots\\
\vdots&0&\ddots&\vdots&\tilde Q_r&\tilde Q_{r-1}&\ddots&\vdots&\\
\vdots&\vdots&\ddots&\vdots&0&\tilde Q_r&\ddots&\vdots\\
\vdots&\vdots&\ddots&\vdots&\vdots&\vdots& \ddots &\vdots \\
0&0&\cdots&\tilde P_s&0&0&\dots&\tilde Q_r\\
\end{pmatrix}.$$
The existence of solutions of equations \eqref{zerocomune} and \eqref{zerocomune2}, is then related to the value of the ${\rm Det}^*$ determinants of the {\em Sylvester} matrices $A(q_2)$ and $B(q_1)$. It is therefore natural to give the following    
\begin{definition}\label{risultante}
With the above notations, the {\em regular resultant} of $P$ and $Q$, with respect to the variable $q_1$, is
\[{\rm Res}(P,Q;q_1):={\rm Det}^*_{n+m}(A(q_2));\]
the {\em regular resultant} of $P$ and $Q$, with respect to the variable $q_2$, is
\[{\rm Res}(P,Q;q_2):={\rm Det}^*_{r+s}(B(q_1)).\]
\end{definition}

\begin{remark}\label{polyrep}
	Thanks to Proposition \ref{Ra} we have that both the regular resultants have a polynomial representative
	in $\mathbb{H}[q_2]$ and $\mathbb{H}[q_1]$ respectively. 
\end{remark}

\noindent In \cite[Definition 13]{Rasheed}, Rasheed gives a definition of resultant
in the framework of 
polynomials in two variables sitting in
Ore algebras which has an
expression similar to ${\rm Res}(P,Q;q_2)$. Due to the chosen setting, 
the author only defines the resultant with respect to the second
variable and studies how to compute it more efficiently.

In our approach, instead, we make use of our constructive definitions and 
investigate the relationships between regular resultants and zeros of the involved polynomials.
Let us begin with an example that enlightens this phenomenon.
\begin{example}\label{esres}
	
	Consider the 
slice regular 	polynomials $$P(q_1,q_2)=(q_1-i)*(q_2-j)$$
	and $$Q(q_1,q_2)=(q_1-i)*(q_2-k).$$ From { Proposition \ref{mlineare}}, we argue that
	these two polynomials vanish on $\{i\}\times \mathbb{C}_i$.

\noindent	To compute the regular resultants, we first consider  
	$$	A(q_2)=\begin{pmatrix}
	-i*(q_2-j)&-i*(q_2-k)\\
	(q_2-j)&(q_2-k)\\
	\end{pmatrix}.$$
Then ${\rm Res}(P,Q;q_1)={\rm Det}_2^*(A(q_2))$ is identically zero as a polynomial in $q_2$.\\ On the other hand,
	$$
	B(q_1)=\begin{pmatrix}
	(q_1-i)*j&(q_1-i)*k\\
	(q_1-i)&(q_1-i)\\
	\end{pmatrix},$$
	thus  ${\rm Res}(P,Q;q_2)={\rm Det}^*_2(B(q_1))=(q_1-i)^{*2}*(k-j)$,
	that vanishes if and only if $q_1=i.$

\noindent Therefore, the common zero locus $\{i\}\times \mathbb{C}_i$  of $P$ and $Q$ 
is contained in the  zero loci of both their regular resultants.
        
\end{example}

The first result for generic slice regular polynomials  is the following.
\begin{lemma}\label{commonleft} 
	Let $P,Q\in \mathbb{H}[q_1,q_2]$. Then
\begin{enumerate}
\item ${\rm Res}(P,Q;q_1)\equiv 0$ if and only if there exist non-zero  $H$ and $K$  in $[\mathbb{H}[q_2]]_{\mathcal {SR}}[q_1]$,
{with $\deg_{q_1}H<\deg_{q_1}Q$ and $\deg_{q_1}K<\deg_{q_1}P$,}
  such that \ $P*H+Q*K\equiv0$;
\item ${\rm Res}(P,Q;q_2)\equiv 0$ if and only if there exist non-zero $\tilde
  H$ and $\tilde K$ in $[\mathbb{H}[q_1]]_{\mathcal{SR}}[q_2]$, {with $\deg_{q_2}\tilde H<\deg_{q_2}Q$ and $\deg_{q_2}\tilde K<\deg_{q_2}P$,} such
  that $P*\tilde H+Q*\tilde K\equiv 0$.
	
\end{enumerate} 
\end{lemma}
\begin{proof} 
  Since we reduced the existence of non-zero solutions of equation $P*H+Q*K\equiv 0$
  to the existence of a non-vanishing solution of the homogeneous
  linear system \eqref{linearsystem}, the claim $(1)$ follows using the
  theory of Dieudonn\'e type determinants introduced in Section \ref{sceD}.
  Indeed there exist non-zero $H$ and $K$ in
  $[\mathbb{H}[q_2]]_{\mathcal{SR}}[q_1]$ such that $P*H+Q*K\equiv 0$ {if and
  only if} the linear system whose associated matrix is $A(q_2)$ has
  ${\rm Det}_{n+m}^*A(q_2)\equiv 0$. By Definition \ref{risultante} we conclude. An
  analogous argument can be used to prove the second equivalence.
\end{proof}

\noindent Let $\langle P,Q\rangle$ denote the right ideal generated by the slice regular polynomials $P$ and $Q$ in $\mathbb{H}[q_1,q_2]$, namely
\[\langle P,Q\rangle :=\{ P*R+Q*S \ : \ R,S\in \mathbb{H}[q_1,q_2]\}.\]
 With the notations introduced above, we can state and prove the following
\begin{theorem} \label{combinazione}
Let $P,Q \in \mathbb{H}[q_1,q_2]$. Then
\begin{enumerate}
	\item the regular resultant ${\rm Res}(P,Q;q_1)$ belongs to $\langle P,Q\rangle\cap \mathbb{H}[q_2]$;
	\item the regular resultant ${\rm Res}(P,Q;q_2)$ belongs to $\langle P,Q\rangle\cap \mathbb{H}[q_1]$.
\end{enumerate}
\end{theorem}
\begin{proof}  
We give the proof for ${\rm Res}(P,Q;q_1),$
  the other case can be managed in the same way. 
  By Definition \ref{risultante}, and thanks to Remark \ref{polyrep}, ${\rm Res}(P,Q;q_1)\in \mathbb{H}[q_2]$.  
 If
  ${\rm Res}(P,Q;q_1)\equiv 0,$ then the claim follows from Lemma \ref{commonleft}. 
  When ${\rm Res}(P,Q;q_1)\not\equiv 0,$ consider the following equation  
  $$P*H+Q*K\equiv 1,$$
  with unknowns $H$ and $K$ in $[\mathbb{H}[q_2]]_{\mathcal {SR}}[q_1]$.
  Arguing as before, we can transform this equation
  in a non-homogeneous linear
  system, analogous to \eqref{linearsystem}, 
  \begin{equation}\label{sysnon}
  	A(q_2)*{\bf x}=(0, \ldots, 0,1)^t.
  	\end{equation}
  Now ${\rm Det}^*_{n+m}(A(q_2))\equiv {\rm Res}(P,Q;q_1)\not\equiv 0.$ 
  Thanks to the adapted Cramer Formula \eqref{cramer}, and to Proposition \ref{Ra}, we can compute a solution ${\bf \hat x}=(\hat H_0,\cdots, \hat H_{m-1},\hat K_0,\cdots,\hat K_{n-1})^t$ of \eqref{sysnon}, where $\hat H_j,\hat K_\ell\in
\mathbb{H}[q_2]*{\rm Res}(P,Q;q_1)^{-*}$ for any $j,\ell$.   
Setting 
\[H_j=\hat H_j*{\rm Res}(P,Q;q_1) \quad  \text{and} \quad K_\ell=\hat{K}_\ell*{\rm Res}(P,Q;q_1)\] 
for any $j,\ell$, we obtain $H(q_1,q_2)=\sum_{j=0}^nq_1^jH_j(q_2)$ and $K(q_1,q_2)=\sum_{\ell=0}^mq_1^\ell K_\ell(q_2)$ such that
\[P*H+Q*K={\rm Res}(P,Q;q_1).\]
\end{proof}
As a direct consequence of the previous result, we show that the phenomenon occurring in Example \ref{esres} is in fact general. 
Indeed, thanks to Proposition \ref{zeri*}, we  immediately 
get the following
\begin{corollary}\label{zeriris}
 
 Let $P,Q$ be two  slice regular polynomials in $\mathbb{H}[q_1,q_2]$.
  If they have a common zero $(a,b)$ such that $ab=ba$, 
then both ${\rm Res}(P,Q;q_1)$ and ${\rm Res}(P,Q;q_2)$
vanish at $(a,b)$ as well. \end{corollary}	

If moreover we know that two slice regular polynomials have a common left (linear) factor, we obtain a stronger result.

\begin{proposition}\label{cor}
  Let $P,Q$ be two slice regular polynomials in $\mathbb{H}[q_1,q_2]$ and
  let $a,b\in \mathbb{H}$.
	\begin{enumerate}
		\item If $P$ and $Q$ have the left common factor $(q_1-a)
		$, then
		${\rm Res}(P,Q;q_1)\equiv 0$;
		\item if $P$ and $Q$ have the left common factor $(q_2-b)$,  then
		${\rm Res}(P,Q;q_2)\equiv 0$.
		
	\end{enumerate}
\end{proposition}

\begin{proof}
Suppose first that $P(q_1,q_2)=(q_1-a)*P_1(q_1,q_2)$ and
$Q(q_1,q_2)=(q_1-a)*Q_1(q_1,q_2)$ with $P_1,Q_1 \in
\mathbb{H}[q_1,q_2]$. Thanks to Theorem \ref{combinazione} we have
that ${\rm Res}(P,Q;q_1)=(q_1-a)*R_1(q_1,q_2)$ for some $R_1\in
\mathbb{H}[q_1,q_2]$. Recalling Proposition \ref{mlineare}, we get
that ${\rm Res}(P,Q;q_1)$ vanishes on $\{a\}\times{C}_a$. Since ${\rm
  Res}(P,Q;q_1)$ is a polynomial in $\mathbb{H}[q_2]$, the Identity
Principle \ref{Id} yields that it is the zero polynomial.

Analogously, suppose now that $P(q_1,q_2)=(q_2-b)*P_2(q_1,q_2)$ and $Q(q_1,q_2)=(q_2-b)*Q_2(q_1,q_2)$ with $P_2,Q_2 \in \mathbb{H}[q_1,q_2]$.
Thanks to Theorem \ref{combinazione} 
we have that ${\rm Res}(P,Q;q_2)=(q_2-b)*R_2(q_1,q_2)$ for some $R_2\in \mathbb{H}[q_1,q_2]$. Recalling Proposition \ref{mlineare}, 
we get that  ${\rm Res}(P,Q;q_2)$	vanishes on
$\mathbb{H}\times \{b\}$. Since ${\rm Res}(P,Q;q_2)$ is a polynomial in $\mathbb{H}[q_1]$, the Identity Principle \ref{Id} yields that it is the zero polynomial.
\end{proof}

\noindent

In the complex setting,   
 the vanishing of one  of the resultants implies the existence of a common factor and this deeply relies on the  fact that $\mathbb{C}[z_1,\ldots,z_n]$ is
        a unique factorization domain.  In our
        setting, we obtain the following result.

\begin{proposition} Let $P$ and $Q$ be in $\mathbb{H}[q_1,q_2]$. If  \ ${\rm Res}(P,Q;q_1)\equiv 0$ or ${\rm Res}(P,Q;q_2)\equiv 0$, then  ${\rm Res}(P^s,Q^s;q_1)\equiv 0$ or, respectively, ${\rm Res}(P^s,Q^s;q_2)\equiv 0$.
\end{proposition}
\begin{proof}
	Suppose ${\rm Res}(P,Q;q_1)\equiv 0$. Then, thanks to Lemma \ref{commonleft}, there exist non-zero $H,K \in \mathbb{H}[q_1,q_1]$ 
	such that $P*H+Q*K\equiv 0,$ {with $\deg_{q_1}H<\deg_{q_1}Q$ and $\deg_{q_1}K<\deg_{q_1}P$.} 
	 Hence
\[P*H\equiv-Q*K, \text{ which, recalling \ref{sym}, implies } P^s*H^s\equiv Q^s*K^s\]
and hence 
\[P^s*H^s-Q^s*K^s \equiv 0.\]

\noindent Now $\deg_{q_1}H^s=2\deg_{q_1}H<2\deg_{q_1}Q=\deg_{q_1}Q^s$ and $\deg_{q_1}K^s=2\deg_{q_1}K<2\deg_{q_1}P=\deg_{q_1}P^s$, and $P^s$ and $Q^s$ have real coefficients. Hence, using the theory of resultants in the commutative setting (see, e.g., \cite{CLO}), we get that
	${\rm Res}(P^s,Q^s;q_1)\equiv 0$.
	The same argument shows that ${\rm Res}(P^s,Q^s;q_2)\equiv 0$.
\end{proof}
Applying again the theory of resultants in the commutative setting (see, e.g., \cite{CLO}) we obtain the following
\begin{corollary}
	Let $P$ and $Q$ be in $\mathbb{H}[q_1,q_2]$. If  \ ${\rm Res}(P,Q;q_1)\equiv 0$ or ${\rm Res}(P,Q;q_2)\equiv 0$, then, for any $I\in \mathbb S$, the restrictions of the symmetrized polynomials $P^s$ and $Q^s$ to the slice $\mathbb C_I$ have a common factor with positive degree in $q_2$ or $q_1$ respectively. 
\end{corollary}

The notion of regular resultant can be also applied to study the
multiplicity of special zeros of slice regular polynomials.

\begin{definition}\label{mult}
Let $P$ be a slice regular polynomial vanishing at $\{a\}\times C_a$. We say
that the {\em multiplicity} of the zero $\{a\}\times C_a$ of $P$ is
$m\in \mathbb N$ if there exists a slice regular polynomial $Q$ which does not
vanish identically on $\{a\}\times C_a$, such that
\[P(q_1,q_2)=(q_1-a)^{*m}* Q(q_1,q_2).\]
Let $P$ be a slice regular polynomial vanishing at
$\mathbb{H}\times\{b\}$. We say that the {\em multiplicity} of the
zero $\mathbb{H}\times\{b\}$ of $P$ is $n\in \mathbb N$ if there
exists a slice regular polynomial $S$ which does not vanish identically on
$\mathbb{H}\times\{b\}$, such that
\[P(q_1,q_2)=(q_2-b)^{*n}* S(q_1,q_2).\]
\end{definition} 

\begin{proposition}
	Let $(a,b)$ with $b \in C_a$, let $P$ be a slice regular
        polynomial vanishing at ${(a,b)}$ and let $P_{q_1}$ denote
        the partial Cullen derivative of $P$ with respect to $q_1$.
If the multiplicity $m$ of $\{a\}\times C_a$ is greater or equal to 2,
then ${\rm Res}(P, P_{q_1}; q_2)\equiv 0$.
        	\end{proposition}
	
\begin{proof}

Recalling Definition \ref{mult}, we can write
		\[P(q_1,q_2)=(q_1-a)^{*m}* Q(q_1,q_2)\]
	with $Q$ a slice regular polynomial which does not vanish at $(a,b)$ for any $b\in C_a$ and $m\in \mathbb{N}$.
	Clearly $P$ belongs to the right ideal generated by $q_1-a$.
	Moreover the partial Cullen derivative with respect to $q_1$ is   
        \[P_{q_1}(q_1,q_2)=m(q_1-a)^{*(m-1)}* Q(q_1,q_2)+(q_1-a)^{*m}*Q_{q_1}(q_1,q_2).\]
Thus $P_{q_1}$ has  $q_1-a$  as a left factor if $m\geq 2$.
Hence, by Proposition \ref{cor}, the resultant ${\rm Res}(P, P_{q_1}; q_2)$ is identically zero.


\end{proof}
\noindent Similar arguments also apply for the case of slice regular polynomials $P$ vanishing identically on $\mathbb{H}\times \{b\}$.	
\begin{proposition}
	Let $(a,b)$ with $b \in C_a$, let $P$ be a slice regular 
	polynomial vanishing at $(a,b)$ and let $P_{q_2}$ denote
	the partial Cullen derivative of $P$ with respect to $q_2$.
	If the multiplicity $n$ of $\mathbb{H}\times \{b\}$ is greater or equal to 2,
	then ${\rm Res}(P, P_{q_2}; q_1)\equiv 0$.
\end{proposition}

For slice regular polynomials in one variable, the resultant involved in these considerations is usually known as {\em discriminant}.

\section*{Acknowledgments}
The authors are partially supported by:
GNSAGA-INdAM via the project ``Hypercomplex function theory and applications''; the first author is also partially supported by MUR project PRIN 2022 ``Real and Complex Manifolds: Geometry and Holomorphic Dynamics'', the second and the third authors are also partially supported by MUR projects PRIN 2022 ``Interactions between Geometric Structures
and Function Theories'' and Finanziamento Premiale FOE 2014 ``Splines for accUrate NumeRics: adaptIve models for Simulation Environments''.

\section*{Declarations}
\noindent The authors have no competing interests to declare that are relevant to the content of this article.\\


\begin{thebibliography}{99}

\bibitem{cileni}  A. E. Almendras Valdebenito--J. A. Briones Donoso--A. L. Tironi ``Resultants of skew polynomials over division rings''
  {\tt  arXiv:2109.13092v1}, (2021)
\bibitem{Ash} H. Aslaksen ``Quaternionic Determinants'' {\sc Math. Intelligencer}, {\bf 18}, p. 57--65, (1996)
\bibitem{bren} J.L. Brenner ``Applications of the Dieudonn\'e Determinant'' {\sc  Linear Algebra App.}, {\bf 1}, p. 511--536, (1968)  
\bibitem{Cas} B. Casselman ``Essays on the structure of reductive groups -- Determinants '' notes last revised in 2021
\bibitem{Cohn} P.M. Cohn {\bf Skew fields --
  Theory of general division rings}
Encyclopedia of Mathematics and its Applications,  Cambridge University Press  {\bf 57} (1995)
\bibitem{CLO} D. Cox-- J. Little--D. O'Shea
{\bf Ideals, Varieties and Algorithms},
Fourth edition
Undergrad. Texts Math.
Springer, Cham, (2015)
\bibitem{Dieudonne} J. Dieudonn\'e,  ``Les d\'eterminants sur un corps non commutatif''  {\sc Bull. Soc. Math. France}, {\bf 71}, p. 27--45 (1943)
\bibitem{erik} A. Lj. Eri\'c ``The resultant of non-commutative polynomials'' {\sc Mat. Vesnik}  {\bf 60},  p. 3--8, (2008)
\bibitem{FGGSS} R.T. Farouki--G. Gentili--C. Giannelli--A. Sestini--C. Stoppato, ``A comprehensive characterization of the set of polynomial curves with rational rotation-minimizing frames'' {\sc Adv. Comput. Math.}, {\bf 43} (1), p. 1--24, (2017)
\bibitem{libroGSS} G. Gentili--C. Stoppato--D.C. Struppa  {\bf Regular Functions of a Quaternionic Variable}, Springer Monographs in Mathematics, (2022)
\bibitem{GST} G. Gentili--C. Stoppato--T. Trinci  ``Zeros of slice functions and polynomials over dual quaternions'' {\sc Trans. A.M.S.} {\bf 374}, p. 5509-5544, (2021)

\bibitem{severalvariables} R. Ghiloni--A. Perotti ``Slice regular functions in several variables'' {\sc Math. Z}  {\bf 302} (1) p. 295--351, (2022)
\bibitem{GSV-LAA}  A. Gori--G. Sarfatti--F. Vlacci ``Zero sets and  Nullstellensatz type theorems for slice regular quaternionic polynomials'', {\tt  arXiv:2212.02301v2} (2023)
\bibitem{cramer2} I. I. Kyrchei, ``Cramer's Rule for Quaternionic Systems of Linear Equations'' {\sc J. Math. Sci. (N.Y.)}, {\bf 155}  (6) p. 839--858, (2008)
\bibitem{lercher} J. Lercher-- H.-P. Schr\"ocker, ``A multiplication technique for the factorization of bivariate quaternionic polynomials'' {\sc Adv. Appl. Clifford Algebr.} {\bf 32}, p. 1--23, (2022).
\bibitem{PV} J. Prezelj--F. Vlacci ``Divergence zero quaternionic vector fields and
Hamming graphs",  {\sc Ars Math. Contemp.} {\bf 19}, p. 189--208, (2020) 
\bibitem{Rasheed} R. Rasheed ``Resultant-based Elimination for Skew Polynomials'', Proceedings of the 23rd International Symposium on Symbolic and Numeric Algorithms for Scientific Computing (SYNASC '21), p. 11-18. IEEE, (2021)
\bibitem{cramer1} V. Retakh--L. Wilson, {\bf Advanced Course on Quasideterminants and Universal Localization}, Centre de Recerca Matematica(CRM), Bellaterra (Barcelona), Spain, 2007, notes of the course.
\bibitem{singularities} C. Stoppato ``Singularities of slice regular functions'' {\sc Math. Nach.} {\bf 285}, p. 1274--1293, (2012)
\bibitem{ZZ}X. Zhao--Y. Zhang ``Resultants of quaternionic polynomials''
  {\sc Hacet. J. Math. Stat.} {\bf 48}, (5), p. 1304--1311, (2019)
  

  
\end{thebibliography}
\end{document}